\documentclass{amsart}

\usepackage{amssymb}
\usepackage{amsthm}
\usepackage{enumerate}
\usepackage{graphicx}
\usepackage{hyperref}
\usepackage{mathrsfs}
\usepackage{amsmath}
\usepackage[english]{babel}
\usepackage{setspace}
\usepackage{color}
\usepackage{mathtools}

\numberwithin{equation}{section}
\newtheorem{mydefinition}{Definition}[section]
\newtheorem{myproposition}{Proposition}[section]
\newtheorem{mybemerkung}[mydefinition]{Remark}
\newtheorem{mysatz}[mydefinition]{Theorem}
\newtheorem{mycorollary}[mydefinition]{Corollary}
\newtheorem{mylemma}[mydefinition]{Lemma}
\DeclareMathOperator{\supp}{supp}
\DeclareMathOperator{\divergence}{div}
\DeclareMathOperator{\sign}{sign}

\begin{document}

\title[Entropy solutions of nonlinear fractional Laplace equations]{Entropy solutions of doubly nonlinear fractional Laplace equations}
%
\author{Niklas Grossekemper}
\address{N. Grossekemper, Universit\"at Duisburg-Essen, Fakult\"at f\"ur Mathematik, Thea-Leymann-Str. 9, 45127 Essen, Germany}
\email{niklas.grossekemper@uni-due.de}
\author{Petra Wittbold}
\address{P. Wittbold, Universit\"at Duisburg-Essen, Fakult\"at f\"ur Mathematik, Thea-Leymann-Str. 9, 45127 Essen, Germany}
\email{petra.wittbold@uni-due.de}
\author{Aleksandra Zimmermann}
\address{A. Zimmermann, Universit\"at Duisburg-Essen, Fakult\"at f\"ur Mathematik, Thea-Leymann-Str. 9, 45127 Essen, Germany}
\email{aleksandra.zimermann@uni-due.de}
\begin{abstract}
In this contribution, we study a class of doubly nonlinear elliptic equations with bounded, merely integrable right-hand side on the whole space $\mathbb{R}^N$. The equation is driven by the fractional Laplacian $(-\Delta)^{\frac{s}{2}}$ for $s\in (0,1]$ and a strongly continuous nonlinear perturbation of first order. It is well known that weak solutions are in genreral not unique in this setting. We are able to prove an $L^1$-contraction and comparison principle and to show existence and uniqueness of entropy solutions.
\end{abstract}
\maketitle
\section{Introduction}
Let $N\in\mathbb{N}$. We consider the doubly nonlinear fractional Laplace equation
\begin{align}
\label{problem}
\tag{$\mathcal{P}_f$}
b(u)+\divergence (F(u))+(-\Delta)^{\frac{s}{2}}u=f \quad \text{in} \ \mathbb{R}^N
\end{align}
where $f\in L^1(\mathbb{R}^N)\cap L^\infty(\mathbb{R}^N)$ and $b:\mathbb{R}\to\mathbb{R}$ is Lipschitz continuous, nondecreasing with $b(0)=0$, satisfying the growth condition $b(u)u\geq\lambda|u|^2$ for some $\lambda>0$. For example, these assumptions are fulfilled by $b=Id+\arctan$, and it is also possible for $b$ to be constant on finite intervals. The function $F:\mathbb{R}\to\mathbb{R}^N$ is locally Lipschitz continuous with $F(0)=0$ and $(-\Delta)^\frac{s}{2}$ is the fractional Laplacian with $s\in (0,1]$ (see Section \ref{FL} for the precise definition).\\
\ \\
The fractional Laplacian is a nonlocal generalization of the classical Laplacian which appears in many fields of analysis and probability theory. In the last two decades, there has been an intensive study of elliptic and evolutionary partial differential equations driven by the fractional Laplacian or related nonlocal operators, see, e.g., \cite{SWZ20} for a list of interesting references. In applications from physics and finance, anomalous diffusion is often modeled by a fractional Laplace evolution equation (see \cite{BV14}, Appendix B for more details and references).\\
In this contribution, we study doubly nonlinear elliptic equations of type \eqref{problem} with bounded right-hand side in $L^1(\mathbb{R}^N)$ on the whole space $\mathbb{R}^N$. The equation is driven by the fractional Laplacian $(-\Delta)^{\frac{s}{2}}$ for $s\in (0,1]$ and a strongly continuous perturbation of first order of the form $\divergence F$ with $F:\mathbb{R}\rightarrow\mathbb{R}^N$ Lipschitz continuous.\\
In \cite{DGV} it was shown that adding a fractional Laplacian with parameter $s\in(1,2)$ to a hyperbolic equation has a smoothing effect, i.e., weak solutions exist and are unique. In \cite{AA}, the fractional Burgers equation was studied for $s\in(0,1)$ and it was shown that weak solutions are not unique. Analogously to the purely hyperbolic case (see \cite{K}), an entropy formulation for fractional scalar conservation laws has been developped in \cite{A}. Consequently, one can not expect well-posedness of weak solutions of \eqref{problem} even in the case $b(u)=u$ and therefore one has to choose a more appropriate solution concept. The well-posedness of \eqref{problem} for $\divergence F\equiv 0$ has been studied in \cite{AAB10} in the framework of renormalized solutions. However, since \eqref{problem} can be interpreted as a special case of a fractional Laplace evolution equation with a first-order convection term, it seems to be more natural to apply the notion of entropy solution in our general setting. In \cite{Betal}, entropy solutions have been introduced for elliptic equations with $L^1$-data. In this contribution, we define the notion of entropy solutions for \eqref{problem}. Moreover, we show existence and the $L^1$-contraction and comparison principles for entropy solutions. In particular, we obtain uniqueness of $b(u)$ in this framework. We recall that existence and $L^1$-contraction allows us to define the $m$-accretive, densely defined, multivalued operator $A_b$ in $L^1(\mathbb{R}^N)$ by
\begin{align*}
&(v,f)\in A_b\Longleftrightarrow\\ 
& v=b(u), u\in L^1(\mathbb{R}^N) \ \text{is entropy solution to} \ \divergence (F(u))+(-\Delta)^{\frac{s}{2}}u=f.
\end{align*}
According to nonlinear semigroup theory (see, e.g., \cite{BCP}), there exists a unique mild solution $b\in C([0,T];L^1(\mathbb{R}^N))$ to the abstract Cauchy problem for $A_b$ for any given data $(v_0,f)\in L^1(\mathbb{R}^N)\times L^1(0,T;L^1(\mathbb{R}^N))$. In the next step, one wants to show that the mild solution is the unique entropy solution to the associated evolution equation
\begin{align}\label{evol}
b(u)_t+\divergence (F(u))+(-\Delta)^{\frac{s}{2}}u&=f \quad \text{in} \ (0,T)\times\mathbb{R}^N\nonumber\\
u(0,\cdot)=u_0
\end{align}
for appropriately chosen data $(u_0,f)\in L^1(\mathbb{R}^N)\times L^1(0,T;L^1(\mathbb{R}^N))$. In this manner, our study of  \eqref{problem} serves as a basis for the investigation of \eqref{evol}, which will be subject of forthcoming work.
\subsection{Outline}
We organize this contribution as follows: We start in section 2 by introducing some basic notations which we will use throughout the paper and give some preliminary results that are used within the later sections. After that, in section 3, we establish the definition of an entropy solution to the equation \eqref{problem} and formulate the two main theorems of this work. In section 4 we prove the $L^1$-contraction and comparison principle with the help of Kruzhkov's method of doubling variables. Since this result is also crucial for the existence proof later on, it is proven before the existence of entropy solutions. Furthermore, the contraction principle gives us uniqueness of the saturation function $b(u)$ and in some cases even the uniqueness of the entropy solution $u$ itself. Finally, in section 5 we prove the existence of entropy solutions. For this, we apply the method of vanishing viscosity to be able to show that there exist weak solutions for a sequence of approximating problems of higher regularity. It is then left to show that these weak solutions converge to the entropy solution of the initial problem. In the appendix, for the sake of completeness, some of the technical results used in this work are proven.

\section{Notations and preliminary results}\label{FL}
Let us introduce some notations and functions that will be frequently used. For any real number $r\in\mathbb{R}$ and $k>0$, we define the functions
\begin{align*}
\sign_0(r)=\begin{cases}
1 \quad &r>0\\
0 \quad &r=0\\
-1 \quad &r<0
\end{cases}
, \ \sign_0^+(r)=\begin{cases}
1 \quad &r>0\\
0 \quad &r\leq0
\end{cases}
, \ T_k(r)=\begin{cases}
k \quad &r>k\\
r \quad &|r|\leq k\\
-k \quad &r<-k \, .
\end{cases}
\end{align*}
	For all $u\in \mathcal{S}(\mathbb{R}^N)$, the Schwartz space of rapidly descreasing functions, and all $s\in(0,2)$, we define the fractional Laplacian $(-\Delta)^\frac{s}{2}$ by
	\begin{align}\label{fracLap}
	(-\Delta)^\frac{s}{2}u(x)&=C(N,s)P.V.\int\limits_{\mathbb{R}^N}\frac{u(x)-u(y)}{|x-y|^{N+s}} \, \mathrm{d}y\nonumber\\
	&=C(N,s)\lim\limits_{\varepsilon\to 0}\int\limits_{\mathbb{R}^N\setminus B_\varepsilon(x)}\frac{u(x)-u(y)}{|x-y|^{N+s}} \, \mathrm{d}y \, ,
	\end{align} 
	where the dimensional constant $C(N,s)>0$ is given by $C(N,s)=\frac{s\Gamma(\frac{N+s}{2})}{2\pi^{\frac{N}{2}+s}\Gamma(1-\frac{s}{2})}$ with $\Gamma$ being the gamma function. The constant $C(N,s)$ is motivated by an equivalent definition of the fractional Laplacian via Fourier transform, i.e. by $(-\Delta)^{\frac{s}{2}}u=\mathcal{F}^{-1}(|\cdot|^s\mathcal{F}(u))$, where it naturally occurs.
We further define the Gagliardo-seminorm by
		\begin{align*}
		[u]_{s/2}\coloneqq\left( \ \int\limits_{\mathbb{R}^N}\int\limits_{\mathbb{R}^N}\frac{|u(x)-u(y)|^2}{|x-y|^{N+s}} \, \mathrm{d}x\mathrm{d}y \right)^{\frac{1}{2}}
		\end{align*}
and the fractional Sobolev space of order $\frac{s}{2}$ by
\begin{align*}
H^{\frac{s}{2}}(\mathbb{R}^N)\coloneqq\{u\in L^2(\mathbb{R}^N) : [u]_{s/2}<\infty\} \, .
\end{align*}
As it is well-known, the fractional Sobolev space is a Hilbert space, if endowed with the natural scalar product which induces the norm
\begin{align*}
\|u\|_{s/2}=(\|u\|_{L^2}^2+[u]_{s/2}^2)^\frac{1}{2} \, .
\end{align*}
We point out that the fractional Sobolev space $H^\frac{s}{2}(\mathbb{R}^N)$ is strictly related to the fractional Laplacian (see \cite{HG} Proposition 3.6).\\
An important tool will be a decomposition of the fractional Laplacian which was introduced by Droniou and Imbert in \cite{DI} Theorem 1, where the authors split the fractional Laplacian into a regular and a singular part.
\begin{myproposition}
	\label{splitprop}
	If $s\in (0,2)$, then for all $u\in \mathcal{S}(\mathbb{R}^N)$, all $r>0$ and all $x\in\mathbb{R}^N$
	\begin{equation}
	\label{fraklap}
	\begin{split}
	(-\Delta)^{\frac{s}{2}}u(x)=&-C(N,s)\int\limits_{\{|z|\geq r\}}\frac{u(x+z)-u(x)}{|z|^{N+s}} \, \mathrm{d}z\\
	&-C(N,s)\int\limits_{\{|z|\leq r\}}\frac{u(x+z)-u(x)-\nabla u(x)\cdot z}{|z|^{N+s}} \, \mathrm{d}z \, .		
	\end{split}
	\end{equation}
\end{myproposition}
\begin{proof}[Proof:]
	\cite{DI}, Theorem 1.
\end{proof}
\begin{mybemerkung}
	With Proposition \ref{splitprop} it is possible to extend the definition of the fractional Laplacian in \eqref{fracLap} for all $u\in C_b^2(\mathbb{R}^N)$. Furthermore, from a well-known nonlocal integration-by-parts formula (see e.g. \cite{CJK} Lemma A.2), we get
	\begin{align}
	\label{symprop}
	\int\limits_{\mathbb{R}^N}(-\Delta)^\frac{s}{2}u(x)\varphi(x) \, \mathrm{d}x = \frac{C(N,s)}{2}\int\limits_{\mathbb{R}^N}\int\limits_{\mathbb{R}^N}\frac{(u(x)-u(y))(\varphi(x)-\varphi(y))}{|x-y|^{N+s}} \, \mathrm{d}x\mathrm{d}y \, ,
	\end{align}
	for any $\varphi\in\mathcal{D}(\mathbb{R}^N)$. The right-hand side can also be associated to a bilinear form which is well-defined on $H^\frac{s}{2}(\mathbb{R}^N)\times H^\frac{s}{2}(\mathbb{R}^N)$. 
\end{mybemerkung}
A useful convergence result, which we will use later on, can also be found in the work of Droniou and Imbert (see \cite{DI}, Proposition 1).
\begin{myproposition}
	\label{lokglm}
	Let $s\in(0,2)$ and $u\in C_b^2(\mathbb{R}^N)$. If $(u_n)_n\subseteq C_b^2(\mathbb{R}^N)$ is bounded in $L^\infty(\mathbb{R}^N)$ and such that $\mathrm{D}^2u_n \to\mathrm{D}^2u$ locally uniformly in $\mathbb{R}^N$ for $n\to\infty$, then
	\begin{align*}
	(-\Delta)^{\frac{s}{2}}u_n\to(-\Delta)^{\frac{s}{2}}u \ \ \ \text{locally uniformly in} \ \mathbb{R}^N \ \text{for} \ n\to\infty \, .
	\end{align*}
\end{myproposition}
\begin{proof}[Proof:]
	\cite{DI}, Proposition 1.
\end{proof}

\section{Concept of solution and main results}
Now, we introduce the notion of entropy solutions which is adapted from the work of N. Alibaud (see \cite{A}) and point out where it fits within the well-known concepts of classical and weak solutions.
	\begin{mydefinition}
\label{definitionentropieloesung}
A function $u\in L^\infty(\mathbb{R}^N)$ is called entropy solution to \eqref{problem} if for all $r>0$, all $\varphi\in \mathcal{D}(\mathbb{R}^N)$ with $\varphi\geq 0$, all $\eta\in C^2(\mathbb{R})$ convex and all $\phi=(\phi_1,\dots,\phi_N)$ with $\phi_i' =\eta'F_i'$ for $i=1,\dots,N$ it holds
\begin{equation}
\label{entropiebedingung}
\begin{split}
&\int\limits_{\mathbb{R}^N}(f(x)-b(u(x)))\eta '(u(x))\varphi(x)+\phi(u(x))\cdot\nabla\varphi(x) \, \mathrm{d}x \\
&+C(N,s)\int\limits_{\mathbb{R}^N}\int\limits_{\{|z|\geq r\}}\eta'(u(x))\frac{u(x+z)-u(x)}{|z|^{N+s}}\varphi(x) \, \mathrm{d}z\mathrm{d}x\\
&+C(N,s)\int\limits_{\mathbb{R}^N}\int\limits_{\{|z|\leq r\}}\eta(u(x))\frac{\varphi(x+z)-\varphi(x)-\nabla\varphi(x)\cdot z}{|z|^{N+s}} \, \mathrm{d}z\mathrm{d}x\geq 0 \, .
\end{split}
\end{equation}
\end{mydefinition}
\begin{mybemerkung}
A function $\eta$, as in Definition \ref{definitionentropieloesung}, is called entropy. The corresponding function $\phi$ is called entropy flux and $(\eta,\phi)$ is called entropy-flux-pair. 
\end{mybemerkung}
	\begin{myproposition}
	\label{klassischeentropieschwacheloesung}
	\ 
	\begin{enumerate}[i.)]
		\item Classical solutions to \eqref{problem}, i.e. $u\in C_b^2(\mathbb{R}^N)$ satisfying the equation \eqref{problem} pointwise for all $x\in \mathbb{R}^N$, are entropy solutions. 
		\item Entropy solutions are weak solutions in the sense that
		\begin{align*}
		\int\limits_{\mathbb{R}^N}(f(x)-b(u(x)))\varphi(x)+F(u(x))\cdot\nabla \varphi(x) - u(x)(-\Delta)^{\frac{s}{2}}\varphi(x) \, \mathrm{d}x = 0
		\end{align*}
		$\forall\varphi\in\mathcal{D}(\mathbb{R}^N)$.
	\end{enumerate}
\end{myproposition}
\begin{proof}[Proof:]
	See Appendix \ref{proof1}.
\end{proof}
The main goal of this paper is to prove the following theorems, concerning the existence and uniqueness of entropy solutions.
\begin{mysatz}
	\label{existenzsatz}
	For all $f\in L^1(\mathbb{R}^N)\cap L^\infty(\mathbb{R}^N)$ there exists an entropy solution $u$ to (\ref{problem}) such that $b(u)\in L^1(\mathbb{R}^N)$.
\end{mysatz}
	\begin{mysatz}
	\label{kontraktionsprinzip}
	Let $f,\tilde{f}\in L^1(\mathbb{R}^N)\cap L^\infty(\mathbb{R}^N)$ and $u,\tilde{u}\in L^\infty(\mathbb{R}^N)$ entropy solutions to \eqref{problem} and $(\mathcal{P}_{\tilde{f}})$ respectively such that $b(u),b(\tilde{u})\in L^1(\mathbb{R}^N).$ Then it holds true that
	\begin{align}
	\int\limits_{\mathbb{R}^N}|b(u(x))-b(\tilde{u}(x))| \, \mathrm{d}x \leq \int\limits_{\mathbb{R}^N}|f(x)-\tilde{f}(x)| \, \mathrm{d}x \, .
	\end{align}
\end{mysatz}
\begin{mybemerkung}
If $f=\tilde{f}$ almost everywhere, it follows directly from Theorem \ref{kontraktionsprinzip} that $b(u)=b(\tilde{u})$ almost everywhere holds. If the nonlinearity $b$ is strictly monotone, we further obtain $u=\tilde{u}$ almost everywhere. Therefore, we know that the entropy solution $u$ of \eqref{problem} is unique. In general, i.e. for just nondecreasing $b$, we can not expect uniqueness of the entropy solution, but only for the saturation function $b(u)$.
\end{mybemerkung}

\section{$L^1$-contraction principle}
In this section, we will prove the $L^1$-contraction principle by applying Kruzhkov's method of doubling variables (see \cite{K}). The $L^1$-contraction principle and $L^1$-comparison principle (see \eqref{vergleichsprinzip}) will play an important role in the proof of the existence theorem.
\begin{proof}[Proof of Theorem \ref{kontraktionsprinzip}:]
	For all $k\in\mathbb{R}$ we choose the entropy-flux-pair $(\eta_k,\phi_k)$, which is defined by
	\begin{align*}
	&\eta_k(a)=|a-k| \, ,\\
	&\phi_k(a)=\int\limits_k^a\sign(\tau-k)F'(\tau) \, \mathrm{d}\tau
	\end{align*}
	for all $a\in\mathbb{R}$. Since these entropies are not smooth enough, we have to show first that we can use them in the entropy inequality (\ref{entropiebedingung}). To this end, we approximate $\eta_k$ in the following way:
	\begin{align*}
	\eta_k^n:\mathbb{R}\to\mathbb{R}, \ \ \ a\mapsto\int\limits_k^a \varrho^n(\sigma-k) \, \mathrm{d}\sigma
	\end{align*}
	with
	\begin{align*}
	\varrho^n:\mathbb{R}\to\mathbb{R}, \ \ \ a\mapsto\begin{cases}
	1 \ &\text{for} \ a>\frac{1}{n}\\
	\sin(na\frac{\pi}{2}) \ &\text{for} \ -\frac{1}{n}\leq a \leq \frac{1}{n}\\
	-1 \ &\text{for} \ a<-\frac{1}{n} \, .
	\end{cases}
	\end{align*}
	Then, for all $n\in\mathbb{N}$, $\eta_k^n\in C^2(\mathbb{R})$ is convex and such that $(\eta_k^n)''$ has compact support. Now we can use them in the entropy inequality \eqref{entropiebedingung} and prove, with the help of Lebesgue's dominated convergence theorem, that $(\eta_k,\phi_k)$ are valid entropy-flux-pairs for all $k\in\mathbb{R}$. We can now apply the method of doubling variables.\\
	For the entropy solution $u$ of \eqref{problem} let $y\in\mathbb{R}^N$ be fixed but arbitrary and choose $\varphi(x)=\varphi_y(x)=\psi(x,y)$ with $\psi\in\mathcal{D}(\mathbb{R}^N\times\mathbb{R}^N)$, $\psi\geq 0$ and $\eta=\eta_k$ with $k=\tilde{u}(y)$. If we apply these in (\ref{entropiebedingung}) and integrate with respect to $y$ over $\mathbb{R}^N$, we get
	\begin{align*}
	&\int\limits_{\mathbb{R}^N}\int\limits_{\mathbb{R}^N}(f(x)-b(u(x)))\sign(u(x)-\tilde{u}(y))\psi(x,y)+\zeta(u(x),\tilde{u}(y))\cdot\nabla_x\psi(x,y) \, \mathrm{d}y\mathrm{d}x \\
	&+C(N,s)\int\limits_{\mathbb{R}^N}\int\limits_{\mathbb{R}^N}\int\limits_{\{|z|\geq r\}}\sign(u(x)-\tilde{u}(y))\frac{u(x+z)-u(x)}{|z|^{N+s}}\psi(x,y) \, \mathrm{d}z\mathrm{d}y\mathrm{d}x\\
	&+C(N,s)\int\limits_{\mathbb{R}^N}\int\limits_{\mathbb{R}^N}\int\limits_{\{|z|\leq r\}}|u(x)-\tilde{u}(y)|\frac{\psi(x+z,y)-\psi(x,y)-\nabla_x\psi(x,y)\cdot z}{|z|^{N+s}} \, \mathrm{d}z\mathrm{d}y\mathrm{d}x\geq 0 \, ,
	\end{align*}
	where the symmetric function $\zeta$ is given by $\zeta(a,b)=F(\max\{a,b\})-F(\min\{a,b\})$. Analogously, for the entropy solution $\tilde{u}$ of $(\mathcal{P}_{\tilde{f}})$  and $x\in\mathbb{R}^N$ fixed but arbitrary, we choose $\varphi(y)=\varphi_x(y)=\psi(x,y)$ as above, $\eta=\eta_k$ with $k=u(x)$, apply these in (\ref{entropiebedingung}) and integrate with respect to $x$ over $\mathbb{R}^N$. If we add these inequalities, we get
	\begin{align*}
	0&\leq\int\limits_{\mathbb{R}^N}\int\limits_{\mathbb{R}^N}(f(x)-\tilde{f}(y))\sign(u(x)-\tilde{u}(y))\psi(x,y)-|b(u(x))-b(\tilde{u}(y))|\psi(x,y)\\
	&+\zeta(u(x),\tilde{u}(y))\cdot(\nabla_x+\nabla_y)\psi(x,y) \, \mathrm{d}y\mathrm{d}x \\
	&+C(N,s)\int\limits_{\mathbb{R}^N}\int\limits_{\mathbb{R}^N}\int\limits_{\{|z|\geq r\}}\frac{|u(x+z)-\tilde{u}(y+z)|-|u(x)-\tilde{u}(y)|}{|z|^{N+s}}\psi(x,y) \, \mathrm{d}z\mathrm{d}y\mathrm{d}x\\
	&+C(N,s)\int\limits_{\mathbb{R}^N}\int\limits_{\mathbb{R}^N}\int\limits_{\{|z|\leq r\}}|u(x)-\tilde{u}(y)|\\
	&\times\frac{\psi(x+z,y)+\psi(x,y+z)-2\psi(x,y)-(\nabla_x+\nabla_y)\psi(x,y)\cdot z}{|z|^{N+s}} \, \mathrm{d}z\mathrm{d}y\mathrm{d}x \, .
	\end{align*}
	In the next step we pass to the limit with $r\to0$.
	\begin{mylemma}
		\label{l1}
		\begin{align*}
		\lim\limits_{r\to 0} \ &C(N,s)\int\limits_{\mathbb{R}^N}\int\limits_{\mathbb{R}^N}\int\limits_{\{|z|\leq r\}}|u(x)-\tilde{u}(y)|\\
		&\times\frac{\psi(x+z,y)+\psi(x,y+z)-2\psi(x,y)-(\nabla_x+\nabla_y)\psi(x,y)\cdot z}{|z|^{N+s}} \, \mathrm{d}z\mathrm{d}y\mathrm{d}x = 0 \, .
		\end{align*}
	\end{mylemma}
	\begin{proof}[Proof:]
		Without loss of generality let $r<1$. We show that the integrand belongs to $L^1(\mathbb{R}^N\times\mathbb{R}^N\times \overline{B_1(0)})$. The claim then follows by Lebesgue's dominated convergence theorem. Similar to the proof of Theorem \ref{klassischeentropieschwacheloesung} $i.)$ we get
		\begin{align*}
		&|u(x)-\tilde{u}(y)|\frac{|\psi(x+z,y)-\psi(x,y)-\nabla_x\psi(x,y)\cdot z|}{|z|^{N+s}}\\
		&\leq (\|u\|_\infty+\|\tilde{u}\|_\infty)\int\limits_0^1(1-\tau)\frac{|\mathrm{D}^2\psi(x+\tau z,y)|}{|z|^{N+s-2}} \, \mathrm{d}\tau\\
		&\leq C(\|u\|_\infty,\|\tilde{u}\|_\infty,\|\mathrm{D}^2(\psi)\|_\infty)\frac{\chi_{\supp(\psi)+\overline{B_1(0,0)}}(x,y)}{|z|^{N+s-2}}\in L^1(\mathbb{R}^N\times\mathbb{R}^N\times \overline{B_1(0)}) \, .
		\end{align*}
		An analogous calculation for the term
		\begin{align*}
		|u(x)-\tilde{u}(y)|\frac{|\psi(x,y+z)-\psi(x,y)-\nabla_y\psi(x,y)\cdot z|}{|z|^{N+s}}
		\end{align*}
		completes the proof.
	\end{proof}
	With similar calculations we can also pass to the limit with $r\to0$ in the regular term of the nonlocality. If we combine these results, we obtain, for $r\to 0$:
	\begin{align*}
	0\leq&\int\limits_{\mathbb{R}^N}\int\limits_{\mathbb{R}^N}(f(x)-\tilde{f}(y))\sign(u(x)-\tilde{u}(y))\psi(x,y)-|b(u(x))-b(\tilde{u}(y))|\psi(x,y)\\
	&+\zeta(u(x),\tilde{u}(y))\cdot(\nabla_x+\nabla_y)\psi(x,y) \, \mathrm{d}y\mathrm{d}x\\
	&+\int\limits_{\mathbb{R}^N}\int\limits_{\mathbb{R}^N}|u(x)-\tilde{u}(y)|\left(C(N,s)\int\limits_{\{|z|\geq 1\}}\frac{\psi(x+z,y+z)-\psi(x,y)}{|z|^{N+s}} \, \mathrm{d}z\right.\\
	&+\left.C(N,s)\int\limits_{\{|z|\leq 1\}}\frac{\psi(x+z,y+z)-\psi(x,y)-(\nabla_x+\nabla_y)\psi(x,y)\cdot z}{|z|^{N+s}} \, \mathrm{d}z \right) \mathrm{d}y\mathrm{d}x \, .
	\end{align*}
	We now choose $\mu>0$, $\psi(x,y)=\rho_\mu(y-x)\Phi(x)$, where $\rho_\mu\in\mathcal{D}(B_\mu(0))$ with $\rho_\mu\geq 0$ such that $\int_{\mathbb{R}^N}\rho_\mu(z) \, \mathrm{d}z=1$ and $\Phi\in\mathcal{D}(\mathbb{R}^N)$ with $\Phi\geq 0$. Then we have
	\begin{align*}
	0\leq &\int\limits_{\mathbb{R}^N}\int\limits_{\mathbb{R}^N}|f(x)-\tilde{f}(y)|\rho_\mu(y-x)\Phi(x)-|b(u(x))-b(\tilde{u}(y))|\rho_\mu(y-x)\Phi(x)\\
	&+|u(x)-\tilde{u}(y)|\rho_\mu(y-x)\Xi(x) \, \mathrm{d}y\mathrm{d}x\eqqcolon I_\mu \, ,
	\end{align*}
	where $\Xi(x)\coloneqq L|\nabla\Phi(x)|-(-\Delta)^{\frac{s}{2}}\Phi(x)$ and $L$ is the Lipschitz constant of $F$ on the set $[-m,m]$ with $m=\max\{\|u\|_\infty , \|\tilde{u}\|_\infty \}$. Letting $\mu\to0$, we can show that
	\begin{equation}
	\label{letzteungleichung} 0\leq\int\limits_{\mathbb{R}^N}|f(x)-\tilde{f}(x)|\Phi(x)-|b(u(x))-b(\tilde{u}(x))|\Phi(x)+|u(x)-\tilde{u}(x)|\Xi(x) \, \mathrm{d}x
	\end{equation}
	holds. Now, choose $\Phi\in\mathcal{D}(\mathbb{R}^N)$, $0\leq \Phi \leq 1$, such that
	\begin{align*}
	\Phi(x)=\begin{cases}
	1 \ &\text{if} \ \|x\|\leq 1\\
	0 \ &\text{if} \ \|x\|\geq 2
	\end{cases}
	\end{align*}
	Then we define $\Phi_n(x)=\Phi(\frac{x}{n})$ $\forall n\in \mathbb{N}$. It yields $D^2\Phi_n\to 0$ pointwise and locally uniformly in $\mathbb{R}^N$ for $n\to\infty$. For $\Phi=\Phi_n$ in (\ref{letzteungleichung}), we can make use of Proposition \ref{lokglm}, the growth condition of $b$, and Lebesgue's theorem of dominated convergence to get
	\begin{align}
	\label{letztergrenzuebergang}
	\nonumber
	0\leq &\lim\limits_{n\to\infty} \int\limits_{\mathbb{R}^N}|f(x)-\tilde{f}(x)|\Phi_n(x)-|b(u(x))-b(\tilde{u}(x))|\Phi_n(x)\\
	&+|u(x)-\tilde{u}(x)|(L|\nabla\Phi_n(x)|-(-\Delta)^{\frac{s}{2}}\Phi_n(x)) \, \mathrm{d}x\\
	\nonumber
	=&\int\limits_{\mathbb{R}^N}|f(x)-\tilde{f}(x)| - |b(u(x))-b(\tilde{u}(x))| \, \mathrm{d}x \, .
	\end{align}
	This completes the proof of Theorem \ref{kontraktionsprinzip}.
\end{proof}

\subsection{Extensions and Remarks}
\begin{mybemerkung}
	Similar to the proof of Theorem \ref{kontraktionsprinzip}, we can show the following $L^1$-comparison principles:
	\begin{align}
	\label{vergleichsprinzip}
	\int\limits_{\mathbb{R}^N}(b(u(x))-b(\tilde{u}(x)))^+ \, \mathrm{d}x \leq \int\limits_{\mathbb{R}^N}(f(x)-\tilde{f}(x))^+ \, \mathrm{d}x
	\end{align}
	and
	\begin{align}
	\label{vergleichsprinzip2}
	\int\limits_{\mathbb{R}^N}(b(u(x))-b(\tilde{u}(x)))^- \, \mathrm{d}x \leq \int\limits_{\mathbb{R}^N}(f(x)-\tilde{f}(x))^- \, \mathrm{d}x \, .
	\end{align}
	To prove this, we apply the method of doubling variables again, but with entropies
	\begin{align*}
	\eta_k(a)=(a-k)^+
	\end{align*}
	and
	\begin{align*}
	\eta_k(a)=(a-k)^- \, .
	\end{align*}
	respectively.
\end{mybemerkung}
\begin{mybemerkung}
	\label{bemerkung111}
	\begin{enumerate}[i.)]
		\item Let $f,\tilde{f},u,\tilde{u}\in L^\infty(\mathbb{R}^N)$ be such that $u$ satisfies the entropy inequality \eqref{entropiebedingung} of \eqref{problem} and $\tilde{u}$ of ($\mathcal{P}_{\tilde{f}}$) respectively. Then we can still prove the \glqq local\grqq \ inequality \eqref{letzteungleichung} for the contraction and comparison principle. 
		\item If $u$ is an entropy solution to \eqref{problem} with $b(u)\in L^1(\mathbb{R}^N)$ and $k>0$ then $(u-k)^+\in L^1(\mathbb{R}^N)$ and, since $\tilde{u}\equiv k$ is a classical solution to $(\mathcal{P}_{\tilde{f}})$ with $\tilde{f}\equiv b(k)$, we can also pass to the limit in \eqref{letztergrenzuebergang} for the comparison principle, to show that \eqref{vergleichsprinzip} holds.
	\end{enumerate}
\end{mybemerkung}
We are now able to prove the following $L^\infty$-estimate.
\begin{mylemma}
	Let $f\in L^1(\mathbb{R}^N)\cap L^\infty(\mathbb{R}^N)$ and $u$ be an entropy solution to \eqref{problem}, then we have
	\begin{align}
	\label{lunendlichabschaetzung}
	\|b(u)\|_\infty\leq\|f\|_\infty \, .
	\end{align}
\end{mylemma}
\begin{proof}[Proof:]
	Since $b$ satisfies the growth condition, $b$ is surjective, i.e., for $\|f^+\|_\infty\in\mathbb{R}$ there exists $c\in\mathbb{R}, \, c\geq0$, such that $b(c)=\|f^+\|_\infty$. Let $\tilde{f}\equiv\|f^+\|_\infty$ and $\tilde{u}\equiv c$, then it follows that $b(u)\leq \|f^+\|_\infty$ almost everywhere with Remark \ref{bemerkung111} ii.). Analogously we can show that $-\|f^-\|_\infty\leq b(u)$ almost everywhere which completes the proof.
\end{proof}

\section{Existence of entropy solutions}
		\subsection{The vanishing viscosity method}
We are now turning our attention to the existence proof for entropy solutions to problem \eqref{problem}. For this, consider the following modified problems for $\varepsilon>0$:
\begin{align}
\label{zup}
\tag{$\tilde{\mathcal{P}}_f^{\varepsilon}$}
b(u)+\varepsilon u-\varepsilon\Delta u+(-\Delta)^\frac{s}{2}u=f
\end{align}
with $f\in L^2(\mathbb{R}^N)$. Assume $u\in C_b^2(\mathbb{R}^N)\cap H^1(\mathbb{R}^N)$ to be a classical solution to \eqref{zup}. We obtain a weak formulation of \eqref{zup} by multiplication with a test function $\varphi\in\mathcal{D}(\mathbb{R}^N)$ and subsequent integration over $\mathbb{R}^N$.
\begin{mydefinition}
	Let $f\in L^2(\mathbb{R}^N)$. A function $u\in H^1(\mathbb{R}^N)$ is called weak solution to \eqref{zup}, if
	\begin{align*}
	&\int\limits_{\mathbb{R}^N}b(u(x))\varphi(x) \, \mathrm{d}x + \varepsilon \int\limits_{\mathbb{R}^N}u(x)\varphi(x) \, \mathrm{d}x + \varepsilon\int\limits_{\mathbb{R}^N} \nabla u(x)\nabla \varphi(x) \, \mathrm{d}x\\ &+\frac{C(N,s)}{2}\int\limits_{\mathbb{R}^N}\int\limits_{\mathbb{R}^N}\frac{(u(x)-u(y))(\varphi(x)-\varphi(y))}{|x-y|^{N+s}} \, \mathrm{d}x\mathrm{d}y = \int\limits_{\mathbb{R}^N}f(x)\varphi(x) \, \mathrm{d}x
	\end{align*}
	for all $\varphi\in H^1(\mathbb{R}^N)$.
\end{mydefinition}
\begin{mybemerkung}
	Since $H^1(\mathbb{R}^N)\hookrightarrow H^{\frac{s}{2}}(\mathbb{R}^N)$ (see \cite{D} Corollaire 4.34 (ii)), all integrals are well-defined.
\end{mybemerkung}
First, we want to prove the existence and uniqueness of weak solutions to the modified problem \eqref{zup}, with the help of Zarantonello's theorem (see \cite{E}, Theorem 3.5.2). By a fixed-point argument, we will then prove that there exist also weak solutions to the modified, doubly nonlinear problem which we will define later on. These solutions turn out to be of higher regularity and are therefore entropy solutions. Finally, we will show that the weak solutions to the modified, doubly nonlinear problems converge to the entropy solution of problem \eqref{problem}.
\begin{myproposition}
	\label{lax}
	For all $f\in L^2(\mathbb{R}^N)$ there exists a unique weak solution $u\in H^1(\mathbb{R}^N)$ to \eqref{zup}.
\end{myproposition}
\begin{proof}[Proof:]
	We define
	\begin{align*}
	a:H^1(\mathbb{R}^N)\times H^1(\mathbb{R}^N)&\to\mathbb{R}\\
	(u,v)&\mapsto \int\limits_{\mathbb{R}^N}b(u(x))v(x) \, \mathrm{d}x + \varepsilon\int\limits_{\mathbb{R}^N}u(x)v(x) \, \mathrm{d}x + \varepsilon\int\limits_{\mathbb{R}^N}\nabla u(x)\cdot\nabla v(x) \, \mathrm{d}x\\
	&+ \frac{C(N,s)}{2}\int\limits_{\mathbb{R}^N}\int\limits_{\mathbb{R}^N}\frac{(u(x)-u(y))(v(x)-v(y))}{|x-y|^{N+s}} \, \mathrm{d}x\mathrm{d}y \, .
	\end{align*}
	Then $a(u,\cdot)$ is linear and bounded for every fixed $u\in H^1(\mathbb{R}^N)$, i.e., it is an element in $H^{-1}(\mathbb{R}^N)$. Consider now
	\begin{align*}
	\mathcal{A}:H^1(\mathbb{R}^N)&\to H^{-1}(\mathbb{R}^N)\\
	u&\mapsto a(u,\cdot)
	\end{align*}
	and we claim that $\mathcal{A}$ is Lipschitz continuous and strongly monotone. It then follows by the Theorem of Zarantonello that $\mathcal{A}$ is bijective and therefore, for all $f\in L^2(\mathbb{R}^N)\hookrightarrow H^{-1}(\mathbb{R}^N)$, there exists a unique $u\in H^1(\mathbb{R}^N)$ such that
	\begin{align*}
	a(u,v)=<f,v>_{H^{-1},H^1}=\int\limits_{\mathbb{R}^N}f(x)v(x) \, \mathrm{d}x
	\end{align*}
	for all $v\in H^1(\mathbb{R}^N)$. This completes the proof.\\
	Let $u_1,u_2\in H^1(\mathbb{R}^N)$. Then it holds
	\begin{align*}
	&\|\mathcal{A}u_1-\mathcal{A}u_2\|_{H^{-1}}\\
	=&\sup\limits_{\|v\|_{H^1}\leq 1} |<\mathcal{A}u_1-\mathcal{A}u_2,v>_{H^{-1},H^1}|\\
	\leq&\sup\limits_{\|v\|_{H^1}\leq 1}\int\limits_{\mathbb{R}^N}|b(u_1(x))-b(u_2(x))||v(x)| \, \mathrm{d}x\\
	&+ \varepsilon\int\limits_{\mathbb{R}^N}|u_1(x)-u_2(x)||v(x)| \, \mathrm{d}x + \varepsilon\int\limits_{\mathbb{R}^N}|\nabla u_1(x)-\nabla u_2(x)||\nabla v(x)| \, \mathrm{d}x\\
	&+ \frac{C(N,s)}{2}\int\limits_{\mathbb{R}^N}\int\limits_{\mathbb{R}^N}\frac{|(u_1(x)-u_2(x))-(u_1(y)-u_2(y))||v(x)-v(y)|}{|x-y|^{N+s}} \, \mathrm{d}x\mathrm{d}y \\
	\leq&\sup\limits_{\|v\|_{H^1}\leq 1}L_b\|u_1-u_2\|_{L^2}\|v\|_{L^2}+\varepsilon\|u_1-u_2\|_{L^2}\|v\|_{L^2}+\varepsilon\|\nabla u_1-\nabla u_2\|_{L^2}\|\nabla v\|_{L^2}\\
	&+\frac{C(N,s)}{2}[u_1-u_2]_{s/2}[v]_{s/2}\\
	\leq&C(N,s,\varepsilon)\|u_1-u_2\|_{H^1} \, ,
	\end{align*}
	where $L_b$ is the Lipschitz constant of $b$ and therefore, $\mathcal{A}$ is Lipschitz continuous. Moreover,
	\begin{align*}
	&<\mathcal{A}u_1-\mathcal{A}u_2,u_1-u_2>_{H^{-1},H^1}\\
	=&\int\limits_{\mathbb{R}^N}(b(u_1(x))-b(u_2(x)))(u_1(x)-u_2(x)) \, \mathrm{d}x\\
	&+ \varepsilon\|u_1-u_2\|_{L^2}^2 + \varepsilon\|\nabla u_1-\nabla u_2\|_{L^2}^2+\frac{C(N,s)}{2}[u_1-u_2]_{s/2}^2\\
	\geq&\varepsilon\|u_1-u_2\|_{H^1}^2 \, ,
	\end{align*}
	since $b$ is nondecreasing, i.e. $\mathcal{A}$ is also strongly monotone.
\end{proof}
\begin{mylemma}
	\label{soboleveinbettung}
	Let $f\in L^2(\mathbb{R}^N)$. For the unique weak solution $u\in H^1(\mathbb{R}^N)$ to \eqref{zup} it holds
	\begin{align*}
	u\in H_{loc}^2(\mathbb{R}^N) \, .
	\end{align*}
\end{mylemma}
\begin{proof}
	In general, for $0<s\leq1$, it is known that (see \cite{HG}, Proposition 3.6)
	\begin{align*}
	u\in H^1(\mathbb{R}^N) \Rightarrow (-\Delta)^{\frac{s}{2}}u \in L^2(\mathbb{R}^N) \, .
	\end{align*}
	As a consequence, for the unique weak solution to \eqref{zup} $u\in H^1(\mathbb{R}^N)$, it follows that
	\begin{align*}
	\varepsilon u - \varepsilon \Delta u = f - (-\Delta)^\frac{s}{2}u - b(u)\in L^2(\mathbb{R}^N)
	\end{align*}
	in the distributional sense, hence $\Delta u\in L^2(\mathbb{R}^N)$ and from classical regularity results for the Laplacian (see e.g. \cite{EV}, Chapter 6.3.1) it follows that
	\begin{align*}
	u\in H_{\text{loc}}^2(\mathbb{R}^N) \, .
	\end{align*}
\end{proof}
\begin{mycorollary}
	\label{corollary}
	From Lemma \ref{soboleveinbettung} it follows that the weak solution $u$ to \eqref{zup} satisfies the equation pointwise almost everywhere.
\end{mycorollary}
Now, we also want to consider the nonlinearity $\divergence F(u)$ in our equation. Since we can not show the existence of weak solutions with such a term directly, we need to look at approximating problems. To this end, for every $R>0$, let $\varrho_R\in C^\infty(\mathbb{R}^N)$ be equipped with the following properties:
\begin{align*}
0&\leq \varrho_R\leq 1 \, ,\\
\varrho_R&=1 \ \text{on} \ B_R(0) \, ,\\
\supp \varrho_R&\Subset B_{2R}(0) \, ,\\
\|\nabla\varrho_R\|_\infty &\leq C \, .
\end{align*}
for a constant $C>0$, independent of $R$. Let us define
\begin{align*}
H^1(\mathbb{R}^N)&\to\mathbb{R}\\
v&\mapsto \int\limits_{\mathbb{R}^N}f(x)v(x) \, \mathrm{d}x + \int\limits_{\mathbb{R}^N} F_\varepsilon(\varrho_R(x)w(x))\nabla(\varrho_R(x)v(x)) \, \mathrm{d}x \, ,
\end{align*}
where $F_\varepsilon=F\circ T_{1/\varepsilon}$ with $T_{1/\varepsilon}$ the truncation function at level $1/\varepsilon$ and $w\in H^1(\mathbb{R}^N)$ fixed, but arbitrary. This map is linear and continuous, i.e. an element of $H^{-1}(\mathbb{R}^N)$.  If we consider the family of integral equations
\begin{align}
\label{huii}
\nonumber
&\int\limits_{\mathbb{R}^N} b(u(x))v(x) \, \mathrm{d}x + \varepsilon \int\limits_{\mathbb{R}^N} u(x)v(x) \, \mathrm{d}x + \varepsilon\int\limits_{\mathbb{R}^N}\nabla u(x) \nabla v(x) \, \mathrm{d}x\\
+&\frac{C(N,s)}{2}\int\limits_{\mathbb{R}^N}\int\limits_{\mathbb{R}^N}\frac{(u(x)-u(y))(v(x)-v(y))}{|x-y|^{N+s}} \, \mathrm{d}x\mathrm{d}y\\
\nonumber
=&\int\limits_{\mathbb{R}^N}f(x)v(x) \, \mathrm{d}x + \int\limits_{\mathbb{R}^N} F_\varepsilon(\varrho_R(x)w(x))\nabla(\varrho_R(x)v(x)) \, \mathrm{d}x
\end{align}
for all $v\in H^1(\mathbb{R}^N)$, we already know from Theorem \ref{lax} that there exists a unique $u=u_{R,w}\in H^1(\mathbb{R}^N)$ which satisfies \eqref{huii} for all $v\in H^1(\mathbb{R}^N)$ and thus, we can define the following map:
\begin{align*}
\Psi_R:H^1(\mathbb{R}^N)&\to H^1(\mathbb{R}^N)\\
w&\mapsto \text{the unique solution} \ u\in H^1(\mathbb{R}^N) \ \text{to} \ \eqref{huii} \, .
\end{align*}
If we take $u$ itself as a test function in \eqref{huii} and exploit the properties of $\varrho_R$, we get the following a-priori estimate:
\begin{align*}
\varepsilon\|u\|_{H^1}^2&\leq \|f\|_{L^2}\|u\|_{L^2}+C_{\varepsilon,R}\|u\|_{H^1}\\
\Rightarrow \|u\|_{H^1}&\leq\frac{\|f\|_{L^2}+C_{\varepsilon,R}}{\varepsilon}\eqqcolon K_{\varepsilon,R} \, .
\end{align*}
Now, consider the set $M=\{u\in H^1(\mathbb{R}^N) : \|u\|_{H^1}\leq K_{\varepsilon,R}\}$ which is nonempty, bounded, closed and convex. If we restrict $\Psi_R$ to the set $M$ and prove that there exists a fixed-point, we obtain the existence of a solution $u$ to \eqref{huii} with $w=u$. To achieve this, we have to prove that $\Psi_R$ is weakly sequentially continuous, for the claim then follows by the fixed-point theorem of Schauder-Tikhonov (see \cite{Z}, Corollary 9.7).\\
\begin{mylemma}
	The map $\Psi_R:M\to M$ is weakly sequentially continuous.
\end{mylemma}
\begin{proof}[Proof:]
	For all $n\in\mathbb{N}$, let $w_n,w\in H^1(\mathbb{R}^N)$ such that $w_n\rightharpoonup w$ in $H^1(\mathbb{R}^N)$ for $n\to\infty$. Since the sequence of solutions $u_n=\Psi_R(w_n)$ is bounded in $H^1(\mathbb{R}^N)$, it is sufficient to prove that every weakly convergent subsequence of $(u_n)_n$ converges weakly to $\Psi_R(w)$. Now, let $(u_n)_n$ be a not relabeled subsequence such that $u_n\rightharpoonup u$ in $H^1(\mathbb{R}^N)$ for $n\to\infty$. By continuous and compact embeddings we can assume, without loss of generality, that $(w_n)_n$ and $(u_n)_n$ converge for $n\to\infty$ in the following way:
	\begin{align}
	\label{blab1}
	&w_n\longrightarrow w \ \ \ \text{in} \ L_\text{loc}^2(\mathbb{R}^N) \ \text{and almost everywhere in} \ \mathbb{R}^N \, ,\\
	\label{blab2}
	&u_n\rightharpoonup u \ \ \ \text{in} \ H^1(\mathbb{R}^N) \hookrightarrow H^\frac{s}{2}(\mathbb{R}^N) \, ,\\
	\label{blab3}
	&u_n\longrightarrow u \ \ \ \text{in} \ L_\text{loc}^2(\mathbb{R}^N) \ \text{and almost everywhere in} \ \mathbb{R}^N \, .
	\end{align}
	Since $u_n$ is a solution to \eqref{huii} with $w=w_n$, for all $n\in\mathbb{N}$ and all $\varphi\in H^1(\mathbb{R}^N)$, we have
	\begin{align*}
	&\int\limits_{\mathbb{R}^N}b(u_n)(x)\varphi(x) \, \mathrm{d}x + \varepsilon\int\limits_{\mathbb{R}^N}u_n(x)\varphi(x) \, \mathrm{d}x + \varepsilon\int\limits_{\mathbb{R}^N}\nabla u_n(x)\nabla \varphi(x) \, \mathrm{d}x\\
	+&\frac{C(N,s)}{2}\int\limits_{\mathbb{R}^N}\int\limits_{\mathbb{R}^N}\frac{(u_n(x)-u_n(y))(\varphi(x)-\varphi(y))}{|x-y|^{N+s}} \, \mathrm{d}x\mathrm{d}y\\
	&=\int\limits_{\mathbb{R}^N}f(x)\varphi(x) \, \mathrm{d}x
	+\int\limits_{\mathbb{R}^N}F_\varepsilon(\varrho_R(x)w_n)(x)\nabla (\varrho_R(x)\varphi(x)) \, \mathrm{d}x \, .
	\end{align*}
	Now we can, thanks to \eqref{blab1}-\eqref{blab3}, pass to the limit in all integrals and obtain $u=\Psi_R(w)$.
\end{proof}
We have therefore proved that, for all $R>0$, there exists $u=u_R\in H^1(\mathbb{R}^N)$ such that
\begin{equation}
\begin{split}
\label{rrr}
&\int\limits_{\mathbb{R}^N} b(u(x))v(x) \, \mathrm{d}x + \varepsilon \int\limits_{\mathbb{R}^N} u(x)v(x) \, \mathrm{d}x + \varepsilon\int\limits_{\mathbb{R}^N}\nabla u(x) \nabla v(x) \, \mathrm{d}x\\
+&\frac{C(N,s)}{2}\int\limits_{\mathbb{R}^N}\int\limits_{\mathbb{R}^N}\frac{(u(x)-u(y))(v(x)-v(y))}{|x-y|^{N+s}} \, \mathrm{d}x\mathrm{d}y\\
=&\int\limits_{\mathbb{R}^N}f(x)v(x) \, \mathrm{d}x + \int\limits_{\mathbb{R}^N} F_\varepsilon(\varrho_R(x)u(x))\nabla(\varrho_R(x)v(x)) \, \mathrm{d}x
\end{split}
\end{equation}
for all $v\in H^1(\mathbb{R}^N)$. Our next goal is to pass to the limit with $R\to\infty$. First, we need a technical result which allows us to get rid of the convection term. From the divergence theorem of Gauß we can show the following Lemma.
\begin{mylemma}
	\label{alekslemma}
	For every $u\in H^1(\mathbb{R}^N)$ it holds
	\begin{align*}
	\int\limits_{\mathbb{R}^N}F_\varepsilon(u(x))\nabla u(x) \, \mathrm{d}x = 0 \, .
	\end{align*}
\end{mylemma}
\begin{proof}[Proof:]
	See Appendix \ref{proof2}
\end{proof}
Hence, if we apply $u=u_R\in H^1(\mathbb{R}^N)$ as a test function in \eqref{rrr} we get
\begin{align*}
\|u_R\|_{H^1}\leq \frac{\|f\|_{L^2}}{\varepsilon} \, ,
\end{align*}
for all $R>0$, where we used Lemma \ref{alekslemma}. Thus, there exists a subsequence, still denoted by $(u_R)_R$, such that
\begin{align*}
u_R&\rightharpoonup u \ \text{in} \ H^1(\mathbb{R}^N) \, ,\\
u_R&\to u \ \text{in} \ L_{loc}^2(\mathbb{R}^N) \ \text{and almost everywhere in} \ \mathbb{R}^N
\end{align*}
for $R\to\infty$. For $v\in C_c^\infty(\mathbb{R}^N)$ we can now pass to the limit with $R\to\infty$ in \eqref{rrr}: If $R$ is large enough such that $\supp v \subseteq B_R(0)$, it yields
\begin{align*}
\int\limits_{\mathbb{R}^N}F_\varepsilon(\varrho_R(x)u_R(x))\nabla(\varrho_R(x)v(x)) \, \mathrm{d}x = \int\limits_{\supp v}F_\varepsilon(u_R(x))\nabla v(x) \, \mathrm{d}x
\end{align*}
and on $\supp{v}\subseteq B_R(0)$ it holds
\begin{align*}
F_\varepsilon(\varrho_R(x)u_R(x))\nabla(\varrho_R(x)v(x)) \to F_\varepsilon(u(x))\nabla v(x)
\end{align*}
almost everywhere for $R\to\infty$ and
\begin{align*}
|F_\varepsilon(\varrho_R(x)u_R(x))\nabla(\varrho_R(x)v(x))|\leq \|F_\varepsilon\|_\infty|\nabla v(x)|\in L^1(\supp{v}) \, .
\end{align*}
Since $H^1(\mathbb{R}^N)=\overline{C_c^\infty(\mathbb{R}^N)}^{\|\cdot\|_{H^1}}$, we can pass to the limit with $R\to\infty$ in \eqref{rrr} for all $v\in H^1(\mathbb{R}^N)$. This means, we have shown the existence of a weak solution to
\begin{equation}
\label{huhu}
\tag{$\mathcal{P}_f^{\varepsilon}$}
b(u)+\varepsilon u-\varepsilon\Delta u + \divergence F_\varepsilon(u) + (-\Delta)^{\frac{s}{2}}u = f
\end{equation}
for all $\varepsilon>0$ and all $f\in L^1(\mathbb{R}^N)\cap L^\infty(\mathbb{R}^N)$.
\begin{myproposition}
	\label{l1absch}
	For the weak solution $u\in H^1(\mathbb{R}^N)$ to \eqref{huhu} the following holds:
	\begin{enumerate}[i.)]
		\label{absch}
		\item $\|b(u)\|_{L^1}\leq \|f\|_{L^1}$, $\|u\|_{L^1}\leq\frac{1}{\lambda}\|f\|_{L^1}$.
		\item $\|u\|_{L^\infty}\leq \frac{1}{\lambda}\|f\|_{L^\infty}$.
	\end{enumerate}
\end{myproposition}
	\begin{proof}[Proof:]
		\begin{enumerate}[i.)]
			\item Let $k>0$. Applying the test function $\frac{1}{k}T_k(u)\in H^1(\mathbb{R}^N)$ in the weak formulation of \eqref{huhu}, we get
			\begin{align*}
			&\int\limits_{\mathbb{R}^N}b(u(x))\frac{1}{k}T_k(u(x)) \, \mathrm{d}x + \frac{\varepsilon}{k}\int\limits_{\mathbb{R}^N}u(x)T_k(u(x)) \, \mathrm{d}x + \frac{\varepsilon}{k}\int\limits_{\{|u|\leq k\}} |\nabla u(x)|^2 \, \mathrm{d}x\\
			-&\int\limits_{\{|u|\leq k\}}F_\varepsilon(u(x))\cdot\nabla u(x) \, \mathrm{d}x +\frac{C(N,s)}{2k}\int\limits_{\mathbb{R}^N}\int\limits_{\mathbb{R}^N}\frac{(u(x)-u(y))(T_k(u(x))-T_k(u(y)))}{|x-y|^{N+s}} \, \mathrm{d}x\mathrm{d}y\\
			=&\int\limits_{\mathbb{R}^N}f(x)\frac{1}{k}T_k(u(x)) \, \mathrm{d}x \, .
			\end{align*}
			If we use Lemma \ref{alekslemma} and the positivity of the terms on the left-hand side, we obtain
			\begin{align*}
			\int\limits_{\mathbb{R}^N}b(u(x))\frac{1}{k}T_k(u(x)) \, \mathrm{d}x  \leq \|f\|_{L^1} \, .
			\end{align*}
			Since for the test function it holds $\frac{1}{k}T_k(u)\to \sign u$ almost everywhere in $\mathbb{R}^N$ for $k\to0$, with Fatou's Lemma we get that
			\begin{align*}
			\|b(u)\|_{L^1}&=\int\limits_{\mathbb{R}^N}|b(u(x))| \, \mathrm{d}x = \int\limits_{\mathbb{R}^N}b(u(x))\sign b(u(x)) \, \mathrm{d}x\\
			&\leq \liminf\limits_{k\to0}\int\limits_{\mathbb{R}^N}b(u(x))\frac{1}{k}T_k(u(x)) \, \mathrm{d}x \leq \|f\|_{L^1} \, .
			\end{align*}
			This implies $b(u)\in L^1(\mathbb{R}^N)$ and thanks to the growth condition of $b$ also $u\in L^1(\mathbb{R}^N)$.
			\item Let now $k,l>0$. This time, we use the test function $\frac{1}{k}T_k^+(u-l)\in H^1(\mathbb{R}^N)$ in the weak formulation of \eqref{huhu}. We then get
			\begin{align*}
			&\int\limits_{\mathbb{R}^N}b(u(x))\frac{1}{k}T_k^+(u(x)-l) \, \mathrm{d}x + \frac{\varepsilon}{k}\int\limits_{\mathbb{R}^N}u(x)T_k^+(u(x)-l) \, \mathrm{d}x\\
			+&\frac{\varepsilon}{k}\int\limits_{\{l<u< l+k\}} |\nabla u(x)|^2 \, \mathrm{d}x-\int\limits_{\{l<u< l+k\}}F_\varepsilon(u(x))\cdot\nabla u(x) \, \mathrm{d}x\\
			+&\frac{C(N,s)}{2k}\int\limits_{\mathbb{R}^N}\int\limits_{\mathbb{R}^N}\frac{(u(x)-u(y))(T_k^+(u(x)-l)-T_k^+(u(y)-l))}{|x-y|^{N+s}} \, \mathrm{d}x\mathrm{d}y\\
			=&\int\limits_{\mathbb{R}^N}f(x)\frac{1}{k}T_k^+(u(x)-l) \, \mathrm{d}x \, .
			\end{align*}
			With the positivity of the integrands, Lemma \ref{alekslemma} and the growth condition of $b$, we obtain
			\begin{align*}
			&\lambda\int\limits_{\mathbb{R}^N}u(x)\frac{1}{k}T_k^+(u(x)-l) \, \mathrm{d}x \leq \int\limits_{\mathbb{R}^N}f(x)\frac{1}{k}T_k^+(u(x)-l) \, \mathrm{d}x\\
			\Rightarrow &\lambda\int\limits_{\mathbb{R}^N}(u(x)-l)\frac{1}{k}T_k^+(u(x)-l) \, \mathrm{d}x \leq \int\limits_{\mathbb{R}^N}(f(x)-\lambda l)\frac{1}{k}T_k^+(u(x)-l) \, \mathrm{d}x \, .
			\end{align*}
			Let $l\geq \frac{\|f^+\|_\infty}{\lambda}$, we then get for $k\to 0$:
			\begin{align*}
			\lambda\int\limits_{\mathbb{R}^N}(u(x)-l)^+ \, \mathrm{d}x \leq \int\limits_{\mathbb{R}^N}(f(x)-\lambda l)\operatorname{sign}^+ (u(x)-l) \, \mathrm{d}x \leq 0 \, .
			\end{align*}
			Since the left integrand is positive, it follows that $(u-l)^+ =0$ almost everywhere in $\mathbb{R}^N$ and therefore $u\leq l$ almost everywhere in $\mathbb{R}^N$. Analogously, we can show that there exists $\tilde{l}>0$ such that $-\tilde{l}\leq u$ almost everywhere in $\mathbb{R}^N$. Thus, the claim follows.
		\end{enumerate}
	\end{proof}

\begin{mybemerkung}
For any $\varepsilon>0$, the unique weak solution $u_\varepsilon$ of \eqref{huhu} is in $H^1(\mathbb{R}^N)$, thus $b(u_{\varepsilon})$ is in $L^2(\mathbb{R}^N)$. Moreover, since $F_{\varepsilon}=(F^1_{\varepsilon},\ldots,F^N_{\varepsilon})$ is Lipschitz continuous, from the chain rule for Sobolev functions it follows that 
\[-\operatorname{div}\,F_{\varepsilon}(u_{\varepsilon})=\sum_{i=1}^N (F^i_{\varepsilon})'(u_{\varepsilon})\frac{\partial u_{\varepsilon}}{\partial x_i}\in L^2(\mathbb{R}^N).\]
With the same arguments as in the proof of Lemma \ref{soboleveinbettung} it follows that $(-\Delta)^\frac{s}{2}u_{\varepsilon}\in L^2(\mathbb{R}^N)$. Therefore,
\[-\Delta u_{\varepsilon}=-u_{\varepsilon}-\frac{1}{\varepsilon}\left[b(u_{\varepsilon})+\operatorname{div}\,(F_{\varepsilon}(u_{\varepsilon}))+(-\Delta)^\frac{s}{2}u_{\varepsilon}\right]\]
in $L^2(\mathbb{R}^N)$ and we may conclude that $u_{\varepsilon}\in H^2_{\operatorname{loc}}(\mathbb{R}^N)$. 
\end{mybemerkung}

\begin{mylemma}
For any $\varepsilon>0$, the unique weak solution $u_\varepsilon$ of \eqref{huhu} satisfies the entropy inequality for all convex entropies $\eta\in C^2(\mathbb{R})$, $\phi_{\varepsilon}=(\phi^1_{\varepsilon},\dots,\phi_{\varepsilon}^N)$ with $(\phi_{\varepsilon}^i)' =\eta'(F_{\varepsilon}^i)'$ for $i=1,\dots,N$, all $\varphi\in\mathcal{D}(\mathbb{R}^N)$ with $\varphi\geq0$, and all $r>0$:
\begin{equation}
\begin{split}
\label{entroop}
&\int\limits_{\mathbb{R}^N}(f(x)-\varepsilon u_\varepsilon(x)-b(u_\varepsilon(x)))\eta '(u_\varepsilon(x))\varphi(x) \, \mathrm{d}x-\varepsilon\int\limits_{\mathbb{R}^N}\nabla(\eta(u_\varepsilon(x))\cdot\nabla\varphi(x) \, \mathrm{d}x\\
+&C(N,s)\int\limits_{\mathbb{R}^N}\int\limits_{\{|z|\geq r\}}\eta'(u_\varepsilon(x))\frac{u_\varepsilon(x+z)-u_\varepsilon(x)}{|z|^{N+s}}\varphi(x) \, \mathrm{d}z\mathrm{d}x\\
+&C(N,s)\int\limits_{\mathbb{R}^N}\int\limits_{\{|z|\leq r\}}\eta(u_\varepsilon(x))\frac{\varphi(x+z)-\varphi(x)-\nabla\varphi(x)\cdot z}{|z|^{N+s}} \, \mathrm{d}z\mathrm{d}x\\
+&\int\limits_{\mathbb{R}^N}\phi_\varepsilon(u_\varepsilon(x))\cdot\nabla \varphi(x) \, \mathrm{d}x\geq 0 \, .
\end{split}
\end{equation}	
\end{mylemma}
\begin{proof}
Since the weak solution $u_\varepsilon$ of \eqref{huhu} is in $H_{loc}^2(\mathbb{R}^N)$, the equation \eqref{huhu} is satisfied pointwise almost everywhere. Similar to the proof of Proposition \ref{klassischeentropieschwacheloesung} i.) (see Appendix \ref{proof1}), we are able to show the following inequality for all convex entropies $\eta\in C^2(\mathbb{R})$, $\phi_{\varepsilon}=(\phi^1_{\varepsilon},\dots,\phi_{\varepsilon}^N)$ with $(\phi_{\varepsilon}^i)' =\eta'(F_{\varepsilon}^i)'$ for $i=1,\dots,N$, all $\varphi\in\mathcal{D}(\mathbb{R}^N)$ with $\varphi\geq0$, and all $r>0$:
\begin{align}
\nonumber
&\int\limits_{\mathbb{R}^N}(f(x)-b(u_\varepsilon(x)))\eta'(u_\varepsilon(x))\varphi(x)+\phi(u_\varepsilon(x))\cdot\nabla\varphi(x) \, \mathrm{d}x\\
\label{dotz1}
&+C(N,s)\int\limits_{\mathbb{R}^N}\int\limits_{\{|z|\geq r\}}\eta'(u_\varepsilon(x))\frac{u_\varepsilon(x+z)-u_\varepsilon(x)}{|z|^{N+s}}\varphi(x) \, \mathrm{d}z\mathrm{d}x\\
\nonumber
&+C(N,s)\int\limits_{\mathbb{R}^N}\int\limits_{\{|z|\leq r\}}\frac{\eta(u_\varepsilon(x+z))-\eta(u_\varepsilon(x))-\nabla\eta(u_\varepsilon(x))\cdot z}{|z|^{N+s}}\varphi(x) \, \mathrm{d}z\mathrm{d}x\geq 0 \, .
\end{align}
It is left to show that we can transfer the fractional derivative in the last term on the left-hand side onto the test function $\varphi$. To this, we can choose a sequence $(u_\varepsilon^n)_n\subseteq C_b^2(\mathbb{R}^N)$ such that
\begin{align*}
u_\varepsilon^n \to u_\varepsilon \quad &\text{in} \ L_{loc}^2(\mathbb{R}^N) \ \text{and a. e. in} \ \mathbb{R}^N \, ,\\
\nabla u_\varepsilon^n \to \nabla u_\varepsilon \quad &\text{in} \ L_{loc}^2(\mathbb{R}^N) \ \text{and a. e. in} \ \mathbb{R}^N \, ,\\
D^2 u_\varepsilon^n \to D^2 u_\varepsilon \quad &\text{in} \ L_{loc}^2(\mathbb{R}^N) \ \text{and a. e. in} \ \mathbb{R}^N \, .
\end{align*}
From Lebesgue's dominated convergence theorem, the proof of Proposition \ref{klassischeentropieschwacheloesung} i.) (see Appendix \ref{proof1}) on the level of the approximating sequence for every $n\in\mathbb{N}$, and again, Lebesgue's dominated convergence theorem, we have
\begin{align*}
&\int\limits_{\mathbb{R}^N}\int\limits_{\{|z|\leq r\}}\frac{\eta(u_\varepsilon(x+z))-\eta(u_\varepsilon(x))-\nabla\eta(u_\varepsilon(x))\cdot z}{|z|^{N+s}}\varphi(x) \, \mathrm{d}z\mathrm{d}x\\
=&\lim\limits_{n\to\infty}\int\limits_{\mathbb{R}^N}\int\limits_{\{|z|\leq r\}}\frac{\eta(u_\varepsilon^n(x+z))-\eta(u_\varepsilon^n(x))-\nabla\eta(u_\varepsilon^n(x))\cdot z}{|z|^{N+s}}\varphi(x) \, \mathrm{d}z\mathrm{d}x\\
=&\lim\limits_{n\to\infty}\int\limits_{\mathbb{R}^N}\int\limits_{\{|z|\leq r\}}\eta(u_\varepsilon^n(x))\frac{\varphi(x+z)-\varphi(x)-\nabla\varphi(x)\cdot z}{|z|^{N+s}} \, \mathrm{d}z\mathrm{d}x\\
=&\int\limits_{\mathbb{R}^N}\int\limits_{\{|z|\leq r\}}\eta(u_\varepsilon(x))\frac{\varphi(x+z)-\varphi(x)-\nabla\varphi(x)\cdot z}{|z|^{N+s}} \, \mathrm{d}z\mathrm{d}x \, .
\end{align*}
\end{proof}
\subsection{Convergence results for the approximating solutions}
In the next step, we want to pass to the limit with $\varepsilon\to0$. We apply the weak solution $u=u_{\varepsilon}\in H^1(\mathbb{R}^N)$ as a test function and use Lemma \ref{alekslemma} to get
\begin{align*}
\varepsilon\|u_{\varepsilon}\|_{L^2}^2&+\int\limits_{\mathbb{R}^N}b(u_{\varepsilon}(x))u_{\varepsilon}(x) \, \mathrm{d}x+\varepsilon\|\nabla u_{\varepsilon}\|_{L^2}^2+\frac{C(N,s)}{2}[u_{\varepsilon}]_{\frac{s}{2}}^2\leq \|f\|_{L^2}\|u_{\varepsilon}\|_{L^2} \, .
\end{align*}
Therefore, we get
\begin{align*}
\lambda\int\limits_{\mathbb{R}^N}|u_\varepsilon(x)|^2 \, \mathrm{d}x \leq \|f\|_{L^2}\|u_{\varepsilon}\|_{L^2}
\end{align*}
and also
\begin{align*}
\varepsilon\|u_{\varepsilon}\|_{L^2}^2+\varepsilon\|\nabla u_{\varepsilon}\|_{L^2}^2+\frac{C(N,s)}{2}[u_{\varepsilon}]_{\frac{s}{2}}^2\leq \|f\|_{L^2}\|u_{\varepsilon}\|_{L^2} \, .
\end{align*}
This implies
\begin{align}
\label{qq}
(u_{\varepsilon})_{\varepsilon} \ &\text{is bounded in} \ L^2(\mathbb{R}^N)\, , \\
\label{ww}
(\sqrt{\varepsilon}u_\varepsilon)_\varepsilon \ &\text{is bounded in} \ L^2(\mathbb{R}^N) \, ,\\
\label{ee}
(\sqrt{\varepsilon}\nabla u_{\varepsilon})_{\varepsilon} \ &\text{is bounded in} \ [L^2(\mathbb{R}^N)]^N \, , \\
\label{rr}
(u_{\varepsilon})_{\varepsilon} \ &\text{is bounded in} \ H^\frac{s}{2}(\mathbb{R}^N) \, . 
\end{align}
From \eqref{qq}-\eqref{rr} we get, for a proper subsequence that is still denoted by $(u_\varepsilon)_\varepsilon$, the following convergence results for $\varepsilon\to0$:
\begin{align}
\label{122}
u_{\varepsilon}\rightharpoonup u \ \ \ &\text{in} \ H^\frac{s}{2}(\mathbb{R}^N) \, ,\\
\label{1224}
\varepsilon u_{\varepsilon}\rightharpoonup 0 \ \ \ &\text{in} \ L^2(\mathbb{R}^N) \, ,\\
\label{1225}
\varepsilon\nabla u_{\varepsilon}\rightharpoonup 0 \ \ \ &\text{in} \ L^2(\mathbb{R}^N) \, ,\\
\label{123}
u_{\varepsilon} \to u \ \ \ &\text{in} \ L_\text{loc}^2(\mathbb{R}^N) \ \text{and a. e. in} \ \mathbb{R}^N \, .
\end{align}
Since we know, from Proposition \ref{absch} ii.), that $(u_\varepsilon)_\varepsilon$ is uniformly bounded in $L^\infty(\mathbb{R}^N)$, we can choose the subsequence such that
\begin{align}
\label{124}
u_{\varepsilon}\rightharpoonup_* u \ \text{in} \ L^\infty(\mathbb{R}^N)
\end{align}
for $\varepsilon\to0$. We now consider the entropy inequality for the approximating solutions $u_\varepsilon$, i.e.
\begin{equation}
\begin{split}
\label{entroop3}
&\int\limits_{\mathbb{R}^N}(f(x)-\varepsilon u_{\varepsilon}(x)-b(u_{\varepsilon}(x)))\eta '(u_{\varepsilon}(x))\varphi(x) \, \mathrm{d}x-\varepsilon\int\limits_{\mathbb{R}^N}\nabla(\eta(u_\varepsilon(x))\cdot\nabla\varphi(x) \, \mathrm{d}x\\
+&C(N,s)\int\limits_{\mathbb{R}^N}\int\limits_{\{|z|\geq r\}}\eta'(u_{\varepsilon}(x))\frac{u_{\varepsilon}(x+z)-u_{\varepsilon}(x)}{|z|^{N+s}}\varphi(x) \, \mathrm{d}z\mathrm{d}x\\
+&C(N,s)\int\limits_{\mathbb{R}^N}\int\limits_{\{|z|\leq r\}}\eta(u_{\varepsilon}(x))\frac{\varphi(x+z)-\varphi(x)-\nabla\varphi(x)\cdot z}{|z|^{N+s}} \, \mathrm{d}z\mathrm{d}x\\
+&\int\limits_{\mathbb{R}^N}\phi_\varepsilon(u_{\varepsilon}(x))\cdot\nabla \varphi(x) \, \mathrm{d}x \geq 0
\end{split}
\end{equation}
and it remains to prove that we can pass to the limit with $\varepsilon\to0$ in \eqref{entroop3} using \eqref{122} - \eqref{124}. We would then obtain $u \in H^\frac{s}{2}(\mathbb{R}^N)\cap L^\infty(\mathbb{R}^N)$ which satisfies the entropy inequality and is therefore the desired entropy solution of problem \eqref{problem} (Notice that $b(u)\in L^1(\mathbb{R}^N)$ follows from Proposition \ref{l1absch} i.), \eqref{123} and Fatou's Lemma).\\
With the above convergence results, we can calculate
\begin{align*}
\int\limits_{\mathbb{R}^N}(f(x)-\varepsilon u_{\varepsilon}(x)-b(u_{\varepsilon}(x)))\eta '(u_{\varepsilon}(x))\varphi(x) \, \mathrm{d}x \to \int\limits_{\mathbb{R}^N}(f(x)-b(u(x)))\eta '(u(x))\varphi(x) \, \mathrm{d}x
\end{align*}
for $\varepsilon\to0$. Furthermore, we have $\varepsilon\nabla(\eta(u_\varepsilon(x)))\cdot\nabla\varphi(x)=\varepsilon\nabla u_\varepsilon(x)\eta'(u_\varepsilon(x))\cdot\nabla\varphi(x)$ and from this, we get for the second integral, together with \eqref{1225}, that
\begin{align*}
\varepsilon\int\limits_{\mathbb{R}^N}\nabla(\eta(u_\varepsilon(x))\nabla\varphi(x) \, \mathrm{d}x \to 0
\end{align*}
for $\varepsilon\to0$. For the nonlocal terms we argue as follows: First we get
\begin{align*}
\eta'(u_{\varepsilon}(x))\frac{u_{\varepsilon}(x+z)-u_{\varepsilon}(x)}{|z|^{N+s}}\varphi(x) \to \eta'(u(x))\frac{u(x+z)-u(x)}{|z|^{N+s}}\varphi(x)
\end{align*}
for almost all $x\in\mathbb{R}^N, z\in B_r(0)^c$ and $\varepsilon\to0$. Further calculations show that
\begin{align*}
\left|\eta'(u_{\varepsilon}(x))\frac{u_{\varepsilon}(x+z)-u_{\varepsilon}(x)}{|z|^{N+s}}\varphi(x) \right| \leq C\frac{|\varphi(x)|}{|z|^{N+s}} \in L^1(\mathbb{R}^N\times B_r(0)^c) \, ,
\end{align*}
for a constant $C>0$, not depending on $\varepsilon$, since $(u_\varepsilon)_\varepsilon$ is uniformly bounded in $L^\infty(\mathbb{R}^N)$. With Lebesgue's dominated convergence theorem, we then get
\begin{align*}
&C(N,s)\int\limits_{\mathbb{R}^N}\int\limits_{\{|z|\geq r\}}\eta'(u_{\varepsilon}(x))\frac{u_{\varepsilon}(x+z)-u_{\varepsilon}(x)}{|z|^{N+s}}\varphi(x) \, \mathrm{d}z\mathrm{d}x\\
\to &C(N,s)\int\limits_{\mathbb{R}^N}\int\limits_{\{|z|\geq r\}}\eta'(u(x))\frac{u(x+z)-u(x)}{|z|^{N+s}}\varphi(x) \, \mathrm{d}z\mathrm{d}x
\end{align*}
for $\varepsilon\to0$. For the singular part
\begin{align*}
\int\limits_{\mathbb{R}^N}\int\limits_{\{|z|\leq r\}}\eta(u_\varepsilon(x))\frac{\varphi(x+z)-\varphi(x)-\nabla\varphi(x)\cdot z}{|z|^{N+s}} \, \mathrm{d}z\mathrm{d}x
\end{align*}
we also get that
\begin{align*}
\eta(u_\varepsilon(x))\frac{\varphi(x+z)-\varphi(x)-\nabla\varphi(x)\cdot z}{|z|^{N+s}} \to \eta(u(x))\frac{\varphi(x+z)-\varphi(x)-\nabla\varphi(x)\cdot z}{|z|^{N+s}}
\end{align*}
for almost all $x\in\mathbb{R}^N, z\in B_r(0)$ and $\varepsilon\to0$. We also notice that $|\varphi(x+z)-\varphi(x)-\nabla\varphi(x)\cdot z|\leq\|D^2\varphi\|_{L^\infty}|z|^2$ for all $x,z\in \mathbb{R}^N$ holds. Therefore, the integrand is dominated by the function
\begin{align*}
C\frac{\|D^2\varphi\|_{L^\infty}}{|z|^{N+s-2}}\chi_{\{|z|\leq r\}}\chi_{\{|x|\leq R+r\}}\in L^1(\mathbb{R}^N\times B_r(0)) \, ,
\end{align*}
for a constant $C>0$, not depending on $\varepsilon$, where $R>0$ is such that $\supp\varphi\subseteq B_R(0)$. Again, by Lebesgue's dominated convergence theorem, we get that
\begin{align*}
\int\limits_{\mathbb{R}^N}\int\limits_{\{|z|\leq r\}}\eta(u_\varepsilon(x))\frac{\varphi(x+z)-\varphi(x)-\nabla\varphi(x)\cdot z}{|z|^{N+s}} \, \mathrm{d}z\mathrm{d}x\\
\to \int\limits_{\mathbb{R}^N}\int\limits_{\{|z|\leq r\}}\eta(u(x))\frac{\varphi(x+z)-\varphi(x)-\nabla\varphi(x)\cdot z}{|z|^{N+s}} \, \mathrm{d}z\mathrm{d}x
\end{align*}
for $\varepsilon\to0$. Since, again, $(u_\varepsilon)_\varepsilon$ is uniformly bounded in $L^\infty(\mathbb{R}^N)$, there exists $\varepsilon_0>0$ such that $\|u_\varepsilon\|_\infty \leq \frac{1}{\varepsilon}$ $\forall0<\varepsilon<\varepsilon_0$. Therefore, we have $T_{1/\varepsilon}(u_\varepsilon)=u_\varepsilon$ $\forall 0<\varepsilon<\varepsilon_0$ and in this case we get
\begin{align*}
\int\limits_{\mathbb{R}^N}\phi_\varepsilon(u_{\varepsilon}(x))\nabla \varphi(x) \, \mathrm{d}x = \int\limits_{\mathbb{R}^N}\phi(u_{\varepsilon}(x))\nabla \varphi(x) \, \mathrm{d}x\to\int\limits_{\mathbb{R}^N}\phi(u(x))\nabla \varphi(x) \, \mathrm{d}x
\end{align*}
for $\varepsilon\to0$. This follows by applying Lebesgue's dominated convergence theorem, since
\begin{align*}
\phi(u_{\varepsilon}(x))=\int\limits_0^{u_\varepsilon(x)}\eta'(\sigma)F'(\sigma) \, \mathrm{d}\sigma \to \int\limits_0^{u(x)}\eta'(\sigma)F'(\sigma) \, \mathrm{d}\sigma=\phi(u(x))
\end{align*}
almost everywhere in $\mathbb{R}^N$ for $\varepsilon\to0$ from \eqref{123}, $\phi(u_{\varepsilon})$ is uniformly bounded in $L^\infty(\mathbb{R}^N)$ and the integral is taken over the compact support of $\varphi$.\\
This completes the proof of Theorem \ref{existenzsatz}.

\section{Appendix}
\begin{proof}[Proof of Proposition \ref{klassischeentropieschwacheloesung}]
	\label{proof1}
	\begin{enumerate}[i.)]
		\item Let $u\in C_b^2(\mathbb{R}^N)$ be such that for all $x\in \mathbb{R}^N$ we have
		\begin{align}
		\label{punktweise}
		b(u(x))+\divergence (F(u(x)))+(-\Delta)^{\frac{s}{2}}u(x)=f(x) \, .
		\end{align}
		Since $\eta\in C^2(\mathbb{R})$ is convex, it holds
		\begin{align*}
		\eta(b)-\eta(a)\geq \eta'(a)(b-a)
		\end{align*}
		for all $a,b\in\mathbb{R}$. Let $a=u(x)$ and $b=u(x+z)$, for $x,z\in\mathbb{R}^N,$ then it follows
		\begin{align*}
		\eta(u(x+z))-\eta(u(x))-\nabla\eta(u(x))\cdot z \geq \eta'(u(x))(u(x+z)-u(x)-\nabla u(x)\cdot z)
		\end{align*}
		and therefore
		\begin{equation}
		\begin{split}
		\label{hui}
		\eta'(u(x))(-\Delta)^{\frac{s}{2}}u(x)=&-C(N,s)\int\limits_{\{|z|\geq r\}}\eta'(u(x))\frac{u(x+z)-u(x)}{|z|^{N+s}} \, \mathrm{d}z\\
		&-C(N,s)\int\limits_{\{|z|\leq r\}}\eta'(u(x))\frac{u(x+z)-u(x)-\nabla u(x)\cdot z}{|z|^{N+s}} \, \mathrm{d}z\\
		\geq&-C(N,s)\int\limits_{\{|z|\geq r\}}\eta'(u(x))\frac{u(x+z)-u(x)}{|z|^{N+s}} \, \mathrm{d}z\\
		&-C(N,s)\int\limits_{\{|z|\leq r\}}\frac{\eta(u(x+z))-\eta(u(x))-\nabla\eta(u(x))\cdot z}{|z|^{N+s}} \, \mathrm{d}z \, .
		\end{split}
		\end{equation}
		If we multiply equation (\ref{punktweise}) pointwise with $\eta'(u(x))$, for all $x\in\mathbb{R}^N$, and use the inequality \eqref{hui}, we get
		\begin{align*}
		&(b(u(x))-f(x))\eta'(u(x))+\divergence (\phi(u(x)))-C(N,s)\int\limits_{\{|z|\geq r\}}\eta'(u(x))\frac{u(x+z)-u(x)}{|z|^{N+s}} \, \mathrm{d}z\\
		&-C(N,s)\int\limits_{\{|z|\leq r\}}\frac{\eta(u(x+z))-\eta(u(x))-\nabla\eta(u(x))\cdot z}{|z|^{N+s}} \, \mathrm{d}z\leq 0 \, ,
		\end{align*}
		where we considered the structure of the entropy flux $\phi'=\eta'F'$ to get $\divergence (\phi(u(x)))=\divergence(F(u(x)))\eta'(u(x))$. If we now multiply this inequality pointwise with $\varphi(x)$, $\varphi\in \mathcal{D}(\mathbb{R}^N)$, $\varphi\geq0$, and then integrate with respect to $x$ over $\mathbb{R}^N$, we get, by applying the integration-by-parts rule to the divergence term, that
		\begin{align}
		\nonumber
		&\int\limits_{\mathbb{R}^N}(f(x)-b(u(x)))\eta'(u(x))\varphi(x)+\phi(u(x))\cdot\nabla\varphi(x) \, \mathrm{d}x\\
		\label{dotz}
		&+C(N,s)\int\limits_{\mathbb{R}^N}\int\limits_{\{|z|\geq r\}}\eta'(u(x))\frac{u(x+z)-u(x)}{|z|^{N+s}}\varphi(x) \, \mathrm{d}z\mathrm{d}x\\
		\nonumber
		&+C(N,s)\int\limits_{\mathbb{R}^N}\int\limits_{\{|z|\leq r\}}\frac{\eta(u(x+z))-\eta(u(x))-\nabla\eta(u(x))\cdot z}{|z|^{N+s}}\varphi(x) \, \mathrm{d}z\mathrm{d}x\geq 0 \, .
		\end{align}
		Now we have to transfer the fractional derivative in the last term of the left-hand side onto the test function $\varphi$. With Taylor's Formula and Fubini's Theorem we get
		\begin{align*}
		&\int\limits_{\mathbb{R}^N}\int\limits_{\{|z|\leq r\}}\frac{\eta(u(x+z))-\eta(u(x))-\nabla\eta(u(x))\cdot z}{|z|^{N+s}}\varphi(x) \, \mathrm{d}z\mathrm{d}x\\
		= &\int\limits_0^1\int\limits_{\mathbb{R}^N}\int\limits_{\{|z|\leq r\}}\frac{(1-\tau)\mathrm{D}^2(\eta(u(x+\tau z))) z\cdot z}{|z|^{N+s}}\varphi(x) \, \mathrm{d}z\mathrm{d}x\mathrm{d}\tau\\
		= &\int\limits_0^1\int\limits_{\{|z|\leq r\}}\frac{(1-\tau)}{|z|^{N+s}}\left( \ \int\limits_{\mathbb{R}^N}\mathrm{D}^2(\eta(u(x+\tau z))) z\cdot(\varphi(x)z) \, \mathrm{d}x\right) \, \mathrm{d}z\mathrm{d}\tau\\
		= &-\int\limits_0^1\int\limits_{\{|z|\leq r\}}\frac{(1-\tau)}{|z|^{N+s}}\left( \ \int\limits_{\mathbb{R}^N}\nabla(\eta(u(x+\tau z)))\cdot z (\nabla\varphi(x)\cdot z) \, \mathrm{d}x\right) \, \mathrm{d}z\mathrm{d}\tau\\
		= &-\int\limits_0^1\int\limits_{\mathbb{R}^N}\int\limits_{\{|z|\leq r\}}\frac{(1-\tau)}{|z|^{N+s}}\nabla(\eta(u(x+\tau z)))\cdot z (\nabla\varphi(x)\cdot z) \, \mathrm{d}z\mathrm{d}x\mathrm{d}\tau\\
		= &-\int\limits_0^1\int\limits_{\mathbb{R}^N}\int\limits_{\{|z|\leq r\}}\frac{(1-\tau)}{|z|^{N+s}}\nabla(\eta(u(x)))\cdot z (\nabla\varphi(x+\tau z)\cdot z) \, \mathrm{d}z\mathrm{d}x\mathrm{d}\tau \, ,
		\end{align*}
		where we used the substitution of variables $(\tau,x,z)\to(\tau,x+\tau z,-z)$. Interchanging the roles of $\eta(u)$ and $\varphi$ in the last equality and calculating the exact steps backwards, we then get
		\begin{align*}
		&\int\limits_{\mathbb{R}^N}\int\limits_{\{|z|\leq r\}}\frac{\eta(u(x+z))-\eta(u(x))-\nabla\eta(u)(x)\cdot z}{|z|^{N+s}}\varphi(x) \, \mathrm{d}z\mathrm{d}x\\
		=&\int\limits_{\mathbb{R}^N}\int\limits_{\{|z|\leq r\}}\eta(u(x))\frac{\varphi(x+z)-\varphi(x)-\nabla\varphi(x)\cdot z}{|z|^{N+s}} \, \mathrm{d}z\mathrm{d}x
		\end{align*}
		and with it the claim.
		\item Let $u$ be an entropy solution of \eqref{problem}. If we first choose $\eta(r)=r$ $\forall r\in\mathbb{R}$, then it follows
		\begin{align*}
		&\int\limits_{\mathbb{R}^N}(f(x)-b(u(x)))\varphi(x)+F(u(x))\cdot\nabla\varphi(x) \, \mathrm{d}x +C(N,s)\int\limits_{\mathbb{R}^N}\int\limits_{\{|z|\geq r\}}\frac{u(x+z)-u(x)}{|z|^{N+s}}\varphi(x) \, \mathrm{d}z\mathrm{d}x\\
		\nonumber
		&+C(N,s)\int\limits_{\mathbb{R}^N}\int\limits_{\{|z|\leq r\}}u(x)\frac{\varphi(x+z)-\varphi(x)-\nabla\varphi(x)\cdot z}{|z|^{N+s}} \, \mathrm{d}z\mathrm{d}x\geq 0
		\end{align*}
		for all $\varphi\in\mathcal{D}(\mathbb{R}^N)$, $\varphi\geq 0$. Analogously, for $\eta(r)=-r$ $\forall r\in\mathbb{R}$, we can show the converse inequality to then get equality. By substituion of variables $(x,z)\mapsto(x+z,-z)$ and Fubini's Theorem, we also get
		\begin{align*}
		\int\limits_{\mathbb{R}^N}\int\limits_{\{|z|\geq r\}}\frac{u(x+z)-u(x)}{|z|^{N+s}}\varphi(x) \, \mathrm{d}z\mathrm{d}x = \int\limits_{\mathbb{R}^N}\int\limits_{\{|z|\geq r\}}u(x)\frac{\varphi(x+z)-\varphi(x)}{|z|^{N+s}} \, \mathrm{d}z\mathrm{d}x \, .
		\end{align*}
		With the representation of the fractional Laplacian in \eqref{fraklap} the claim follows.
	\end{enumerate}
\end{proof}

\begin{proof}[Proof of Lemma \ref{alekslemma}]
	\label{proof2}
	By the divergence theorem of Gauß it is well-known that, on every bounded Lipschitz domain $\Omega\subseteq\mathbb{R}^N$, it holds
	\begin{align*}
	\int\limits_{\Omega}F_\varepsilon(u(x))\nabla u(x) \, \mathrm{d}x = 0
	\end{align*}
	for all $u\in H_0^1(\Omega)$. Let $(u_n)_n\subseteq C_c^\infty(\mathbb{R}^N)$ be such that $\supp u_n\subseteq B_n(0)\eqqcolon\Omega_n$ for every $n\in \mathbb{N}$ and
	\begin{align*}
	u_n\to u \quad \text{in} \ H^1(\mathbb{R}^N) \, .
	\end{align*}
	Therefore, we get
	\begin{align*}
	\int\limits_{\mathbb{R}^N}F_\varepsilon(u(x))\nabla u(x) \, \mathrm{d}x = \lim\limits_{n\to\infty} \int\limits_{\Omega_n}F_\varepsilon(u_n(x))\nabla u_n(x) \, \mathrm{d}x = 0 \, .
	\end{align*}
\end{proof}

\end{document}